\documentclass[10pt,twoside,a4]{amsart}
\usepackage{amssymb,amsmath,amsfonts,amsthm,graphicx,amscd}
\usepackage{enumerate,mathrsfs,longtable,bm,float,xcolor}
\usepackage[dvips]{epsfig}
\usepackage{amsmath,amsfonts,amssymb,amsthm,bm,float,makecell,booktabs,MnSymbol}

\usepackage{enumerate}
\usepackage{amscd}
\usepackage{tikz}
\usetikzlibrary{positioning}
\usepackage[all,cmtip,line]{xy}

\DeclareMathAlphabet{\mathpzc}{OT1}{pzc}{m}{it}

\numberwithin{equation}{section}

\theoremstyle{plain}

\newtheorem{lem}{Lemma}[section]

\newtheorem{thm}[lem]{Theorem}
\newtheorem{prop}[lem]{Proposition}
\newtheorem{cor}[lem]{Corollary}
\theoremstyle{definition}
\newtheorem{exa}[lem]{Example}
\newtheorem{rem}[lem]{Remark}

\newtheorem{defn}[lem]{Definition}

\begin{document}
	
	\baselineskip 13truept
	
	\title{On $S$-Prime Element Principle}
	
\author{Sachin Sarode*,  Chetan Patil** and Vinayak Joshi***}
\address{\rm *Department of Mathematics, Shri Muktanand College\\ Gangapur, Dist. Chh. Sambhajinagar - 431 109, India.} \email{sarodemaths@gmail.com}

\address{\rm **School of Technology Management and Engineering, SVKM NMIMS Global University,
	Dhule-424 001, India.}
\email{patilcs19@gmail.com}
\address{\rm ***Department of Mathematics, Savitribai Phule Pune University, Pune-411 007, India.}
\email{vvjoshi@unipune.ac.in \\
	vinayakjoshi111@yahoo.com }


	\maketitle

	\begin{abstract}
		In this paper, we introduce $S$-prime elements in $V$-lattices, where $S$ is a multiplicatively closed subset of a $V$-lattice $L$. In addition, we introduce $S$-Prime Element Principle to prove that certain elements in  $V$-lattices are
		$S$-prime elements. This principle leads to a direct and uniform approach to the results on the existence of prime elements in multiplicative lattices, when $S=\{1\}$.  
	\end{abstract}
	
	\maketitle
	


	\baselineskip 14truept 
	\section{Introduction}\label{intro}

	Prime ideals play a central role in commutative ring theory, serving as the
	building blocks of algebraic geometry, number theory, and module theory. In fact, there are few results which demonstrates that the behavior of the prime ideals influence the behavior of all ideal. For example, if prime ideals are finitely generated, then all  ideals are finitely generated, equivalently, a ring is a Noetherian ring.

	Hamed and Malek \cite{HM} introduced   a generalization prime ideals known as
	\emph{$S$-prime ideals}. Let $S$ be a multiplicatively closed subset of a commutative ring $R$ with identity.
	A proper ideal $I \subsetneqq R$  disjoint from $S$ is called an \textit{$S$-prime} ideal if there exists
	$s \in S$ such that for all $a,b \in R$, 
	$
	ab \in I$, then $ sa \in I \ \text{ or } \ sb \in I.
	$
	Every prime ideal is an $S$-prime ideal (taking $S=\{1\}$), but in general, an
	$S$-prime ideal need not be prime. This concept broadens the applicability of
	prime ideals and provides new tools for studying ring-theoretic structures in
	contexts where strict primality is too restrictive.
	
	Our goal in this paper is to transport this idea from rings to multiplicative lattices - a concept introduced by Ward and Dilworth \cite{WD} to extend the abstract ideal theory of rings. 
	In particular, we define and study $S$-prime elements in $V$-lattices,
	extending both the lattice-theoretic notion of prime elements and the
	ring-theoretic notion of $S$-prime ideals.

	The study of prime elements, the generalization of prime elements, and their properties in  multiplicative lattice is the  main focus of many researchers. Different classes of elements and generalization of the prime element in multiplicative lattices were  studied by   Burton \cite{B}, Joshi and Ballal \cite{JB}, Jayaram \cite{J}, Jayaram and Johnson \cite{JJ, JJJ},  Sarode \cite{S} and, Sarode and Joshi \cite{SJ}.  
	
	\begin{defn}
		A \textit{multiplicative lattice} is a complete lattice $L$ together with a binary operation (multiplication) ``$\cdot$" satisfying:
		\begin{enumerate}
			\item $(L,\;\cdot)$ is commutative and associative,
			\item multiplication distributes over arbitrary joins,
			\item $1$ is the multiplicative identity.
		\end{enumerate}
	\end{defn}

	If $L$ is a multiplicative lattice and $a, b \in L$, then $a\cdot b \leq a \wedge b$. This property follows from the fact that the multiplication distributes over join and 1 is the multiplicative identity.

	\begin{minipage}{0.73\textwidth}
		\begin{rem}\label{43.2.}
			A natural example of a multiplicative lattice is the ideal lattice of all ideals of a commutative ring $R$ with 1. However, not all lattices admit a multiplicative lattice structure. For example, consider the lattice $N_5 = \{0,a,b,c,1\}$ shown in Figure~\ref{fig:n5}. Suppose, for contradiction, that $N_5$ admits a multiplication satisfying the axioms of a multiplicative lattice. Then
			$ b \cdot (a \vee c) = b \cdot 1 = b. $
			On the other hand,
			$ (b \cdot a) \vee (b \cdot c) \leq a \vee (b \wedge c) \leq a, $
			which implies
			$ b \cdot (a \vee c) > (b \cdot a) \vee (b \cdot c), $
			a contradiction. Hence $N_5$ is not a multiplicative lattice under any multiplication.\end{rem}
	\end{minipage}\hfill 	
	\begin{minipage}{0.2\textwidth}
		\begin{tikzpicture}[scale=0.8]
			\draw (5,0) -- (4,1) node at (5, -0.4) {$0$}; 
			\draw [fill=black] (5,0) circle(.1); 
			\draw (5,0) -- (6,2) node at (6, 2.3) {$c$}; 
			\draw [fill=black] (6,2) circle (.1); 
			\draw (4,1) -- (4,2) node at (3.7, 1) {$a$}; 
			\draw [fill=black](4,1) circle (.1); 
			\draw (4,2) -- (5,3) node at (3.7, 2) {$b$};
			\draw [fill=black](4,2) circle (.1); 
			\draw (6,2) -- (5,3) node at (5, 3.5) {$1$}; 
			\draw [fill=black](5,3) circle (.1);   
			\draw node at (5,-1) {Figure 1: $N_5$};
		\end{tikzpicture}
		
		\label{fig:n5} \vspace{.2in}
	\end{minipage}

	 In the lattice $N_5$, the multiplication do not distribute over join. This motivates us to define a broader class of lattices.

	\begin{defn}\label{def:vlattice}
		A complete lattice $L$ is said to be a \textit{$V$-lattice} if there exists a commutative, associative binary operation ``$\cdot$" on $L$ such that:
		\begin{enumerate}
			\item $a \leq b \implies a \cdot c \leq b \cdot c$, for  $a, b, c\in L$.
			\item $a \cdot b \leq a \wedge b$, for  $a, b\in L$.
			\item $a \cdot 1 = a$ for every $a\in L$.
		\end{enumerate}
	\end{defn}
	
	\begin{rem} \label{1.7.}
		Every multiplicative lattice is a $V$-lattice, but not conversely.
	\end{rem}

	In the literature on commutative rings, theorems concerning the \textit{existence} of prime ideals play a central role. 
	Analogous existence results have been established for prime elements in multiplicative lattices. 
	One such fundamental result, due to D.~D.~Anderson, is stated below for a special class of lattices known as \emph{$c$–lattices}. 
	By a $c$–lattice, we mean a compactly generated multiplicative lattice in which the greatest element $1$ is compact and product of two compact element is compact. 
	
	An element $c\in L$ is said to be \emph{compact} if whenever $c \leq \bigvee_{\alpha\in\Lambda} a_{\alpha}$ for some indexed family $\{a_{\alpha}\}_{\alpha\in\Lambda}\subseteq L$, 
	there exist finitely many indices $\alpha_{1},\dots,\alpha_{n}\in\Lambda$ such that 
	$c \leq \bigvee_{i=1}^{n} a_{\alpha_{i}}$.
	Here $L_{*}$ denotes the set of all compact elements of a complete lattice $L$. A lattice $L$ is \textit{compactly generated} if every
	element of $L$ is a join of compact elements of $L$.
	A proper element $p$ of a $c$–lattice is called \emph{prime} if $a\cdot b \leq p$ implies $a \leq p$ or $b \leq p$ for all $a,b\in L_{*}$.
	A nonempty subset $S$ of $L_{*}$ in a $V$-lattice $L$ is \textit{multiplicatively closed} if $s_{1} \cdot s_{2} \in S$ whenever $s_{1}, s_{2} \in S$.
	
	\textbf{Throughout this paper, $S$ is a multiplicatively closed subset of a $V$-lattice $L$ such that $1 \in S$ and $0 \notin S$.}
	
	\begin{thm}[{D.D.~Anderson~\cite{A}}]\label{2.13.}
		Let $L$ be a $c$-lattice. Suppose $a \in L$ and $t \nleq a$ for all $t \in S$, where $S$ is a multiplicatively closed subset of $L$. Then there exists a prime element $p$ of $L$ such that $a \leq p$, and $p$ is maximal with respect to the property $t \nleq p$ for all $t \in S$.
	\end{thm}
	
	\begin{lem}[{F. Alarcon et al. \cite{AAJ}}]		Let $L$ be a $c$-lattice. Then $\operatorname{Max}(L) \subseteq \operatorname{Spec}(L)$.
	\end{lem}
	
	These results not only parallel the classical theorems of commutative algebra but also highlight the robustness of the lattice-theoretic approach to algebraic structures.
	
	\begin{lem}[{Joshi and Sarode \cite{JS}}]
		Let $L$ be a reduced, $1$-compact, compactly generated lattice. For $x \in L$, if $x^*$ is maximal among $\{a^* \;|\; a \in L, a^* \not= 1\}$, then $x^*$ is a prime element of $L$.
	\end{lem}
	
	All results mentioned above have different proof. Lam and Reyes \cite{lr} introduced the Prime Ideal
	Principle in commutative rings to unify the results about the existence of prime ideals. The
	Principle states that, for suitable ideal family $ \mathscr{F} $,
	every ideal maximal with respect to not being in $ \mathscr{F} $
	is a prime ideal.  On the same line, Kavishwar and Joshi
	\cite{KJ} introduced the Prime Ideal Principle in lattices. 
	
	In this paper, we introduce $S$-prime elements in $V$-lattices, where $S$ is a multiplicatively closed subset of a $V$-lattice $L$. In addition, we introduce $S$-Prime Element Principle to prove that certain elements in  $V$-lattices are
	$S$-prime elements. This principle leads to a direct and uniform approach to the results on the existence of prime elements in multiplicative lattices, when $S=\{1\}$.

	\section{$S$-prime element}
	
	The study of prime ideals and prime elements has been of fundamental importance in both commutative algebra and multiplicative lattice theory. As mentioned earlier, in the context of commutative rings with identity, Hamed and Malek \cite{HM} introduced the notion of an  $S$-prime ideal by associating primeness with a multiplicatively closed subset 
	$S$ of the ring. This generalization provides a natural extension of the classical prime ideal, with the case 
	$S=U(R)$, the set of units of 
	$R$ recovering the usual prime ideals. 
	Motivated by these developments, it is natural to explore corresponding generalizations in the setting of multiplicative lattices. In particular, within a 
	$V$-lattice, the introduction of 
	$S$-prime elements offers a systematic framework for extending the classical notion of prime elements to situations where multiplicatively closed subsets play a central role. This construction not only extends the ring-theoretic theory but also lattice-theoretic framework. In what follows, we formalize the definition of 
	$S$-prime elements in  $V$-lattices, provide illustrative examples, and establish their fundamental properties.	
	
	We introduce the $S$-prime element in $V$-lattices as follows: 
	
	\begin{defn}
Let $S$ be a multiplicatively closed subset of a $V$-lattice $L$.	A proper element $p$ of $L$ satisfying $t \not \leq p$ for all $t \in S$ is called an \textit{$S$-prime  element} if there exists $s \in S$ such that for all $a, b \in L, a \cdot b \leq p$ implies $s \cdot a \leq p$ or $s \cdot b \leq p$.

		If we take $S=\{1\}$, then we get the classical definition of prime elements. That is, 	an element $p \neq 1$ of a $V$-lattice $L$  is \textit{prime} element if $a
		\cdot b \leq p$ implies $a \leq p$ or $b \leq p$ for all $a, ~b \in L$.  		If we consider $L$ to be a $c$-lattice, then this definition of prime element coincides with the following definition. An element $p \neq 1$ of a $c$-lattice $L$  is \textit{prime} element if $a
		\cdot b \leq p$ implies $a \leq p$ or $b \leq p$ for all $a, ~b \in L_*$, where $L_*$ is the set of all compact elements of $L$.

		The set of all $S$-prime elements of $L$ is denoted by $Spec_s(L)$. If $S = \{1\}$, then $Spec_s(L)$ is denoted by $Spec(L)$, the set of all prime elements of $L$.
	\end{defn}
	
	\begin{exa}\label{34.6.} 
		Consider the lattice $N_5$ whose Hasse diagram is shown in
		Figure \ref{fig:n5}.   Note that $N_5$ is a $V$-lattice with the multiplication as meet. If we take $S = \{1 \}$, then $b$ and $c$ are $S$-prime elements of $N_5$, however, $a$ and $0$ are not $S$-prime elements of $N_5$.
	If  $S = \{1, c \}$, then $0, ~ a$ and $b$ are $S$-prime elements of $N_5$, however, $c$ is not an $S$-prime element of $N_5$, as $c\not\leq c\in S$.  Further,  if $S = \{1, ~a \}$, then $0$ and $c$ are $S$-prime elements of $N_5$ but $a$ and $b$ are not  $S$-prime elements of $N_5$. \end{exa}

	\begin{exa}\label{44.6.} 
		Consider the lattice $L=Id(\mathbb{Z}_{12})$, the ideal lattice of the ring $\mathbb{Z}_{12} $.  Observe that  $L$ is a multiplicative lattice with the multiplication as the multiplication  of ideals. If we take $S = \{(1), ~(4) \}$, then $(0), ~ (6)$, and $ (3)$ are $S$-prime elements of $L$, however, $(2)$ and $(4)$ are not $S$-prime elements of $L$. If $S = \{(1), ~(3) \}$, then $(2)$ and $ (6)$ are $S$-prime but $(0), ~(3)$, and $(4)$ are not $S$-prime.
	\end{exa}

	From the above two examples, we have following observations.

	\begin{rem}
		\begin{enumerate}
			\item If $S = \{1\}$ is a multiplicatively closed subset of a $V$-lattice $L$, then prime elements of $L$ and $S$-prime elements of $L$  coincide. 
			
			\item If $p$ is a prime element of a $V$-lattice $L$, then $p$ need not be an $S$-prime element of $L$ for some multiplicatively closed subset $S$ of $L$. In Example \ref{44.6.}, $(2)$ is a prime element of $L$, but $(2)$ is not an $S$-prime element of $L$ with respect to multiplicatively closed subset $S = \{(1), ~(4)\}$. Note that if $p$ is a prime element of  $L$ such that $t \not\leqq p$ for all $t \in S$, then $p$ is an $S$-prime element of $L$ with respect to $S$.
			
			\item If $p$ is an $S$-prime element of a $V$-lattice $L$ with respect to some multiplicatively closed subset $S$, then $p$ need not be a prime element of $L$. In Example \ref{44.6.}, $(6)$ is an $S$-prime element  with respect to multiplicatively closed subset $S = \{(1), ~(4)\}$, but $(6)$ is not a prime element of $L$.
			
			\item If $t$ is any element of multiplicatively closed subset $S$ of $V$-lattice $L$, then $t$ is not  an $S$-prime element of $L$.
			
		\end{enumerate}
	\end{rem}
	

	Let $R$ be a commutative ring with identity and $S$ be a multiplicatively closed subset of $R$. Then $S_L = \{(a) \in Id(R) \;|\; a \in S\}$ is a multiplicatively closed subset of a $V$-lattice $L = Id(R)$,  where $Id(R)$ is  the ideal lattice of $R$. Note that the finitely generated ideals of $R$ are the compact elements of $Id(R)$.
	
	Following result is due to Hamed and Malek \cite{HM}.
	
	\begin{thm}[Hamed and Malek \cite{HM}] \label{1.20.}
		Let $R$ be a commutative ring with identity, $S $ be a multiplicative subset of $R$ and $P$ be an ideal of $R$ disjoint with $S$. Then $P$ is $S$-prime if and only if there exists $s \in S$, such that for all $I, J$ two ideals of $R$, if $IJ \subseteq P$, then $sI \subseteq P$ or $sJ \subseteq P$.
	\end{thm}

	\begin{thm}
		Let $R$ be a commutative ring with identity and $S$ a multiplicative subset of $R$. 
		Let $L=Id(R)$ denote the multiplicative lattice of all ideals of $R$, and define 
		$S_L=\{(a)\in L \mid a\in S\}.$
		Then an ideal $P$ of $R$ is an $S$-prime ideal of $R$ if and only if $P$ (viewed as an element of $L$) is an $S_L$-prime element of $L$.
	\end{thm}
	
	\begin{proof}
		Suppose $P$ is an $S$-prime ideal of $R$. Then $S\cap P=\varnothing$.  We claim that $t\not\leqq P$ for all $t\in S_L$. Suppose on the contrary that  there exists $t=(a)\in S_L$ with $t\le P$. But then $a\in P\cap S$, a contradiction.  
		Hence $t\not\leqq P$ for all $t\in S_L$.  
		
		Now, let $I,J\in L$ with $IJ \leq P$, i.e., $IJ\subseteq P$. Since $P$ is $S$-prime, there exists $s\in S$ such that $sI\subseteq P$ or $sJ\subseteq P$.  
		Equivalently, $(s)I\leq  P$ or $(s)J\leq P$ in $L$.  
		Thus $P$ is an $S_L$-prime element of $L$.
		
		Conversely, assume that $P$ is an $S_L$-prime element of $L$.  
		By the definition, $t\not\leq P$ for every $t\in S_L$. If some $a\in S\cap P$, then $(a)\in S_L$ with $(a)\le P$, a contradiction. Hence $S\cap P=\varnothing$.  
		
		Further, since $P$ is $S_L$-prime, there exists $(s)\in S_L$ (for some $s\in S$) such that for all $I,J\in L$,
		$IJ\leq  P$ implies that $ (s)I\leq P \ \text{ or }\ (s)J\leq P$.
		That is, $sI\subseteq P$ or $sJ\subseteq P$. Thus $P$ is an $S$-prime ideal of $R$.
	\end{proof}
	
	Thus, the lattice-theoretic notion of $S$-prime elements in $L$  extends the ring-theoretic notion of $S$-prime ideals in $R$.

	\begin{defn} [ \cite{D, SJ, WD}]
		Let $a,b $ be any elements of    a multiplicative lattice $L$. Then $(a:b) =
		\bigvee \{x \;|\;  x\cdot b \leq a\}$. Note that $x \cdot b \leq a
		\Leftrightarrow x \leq (a:b)$.

	\end{defn}

	\begin{lem} \label{5.11.}
		Let $S$ be a multiplicatively closed subset of a multiplicative lattice $L$. If $p$ is proper element of $L$ such that $t \not\leqq p$ for all $t \in S$. Then $p$ is an $S$-prime element of $L$ if and only if $(p : s)$ is a prime element of $L$ for some $s \in S$. 
	\end{lem}
	
	\begin{proof}
		 First,  we show that $(p:s)$ is a proper element of $L$ for all $s\in S$. Suppose on the contrary that  $(p:s) = 1$ for some $s\in S$. We always  have   $(p:s) \cdot s \leq p$. This inequality with $(p:s) = 1$, gives $s \leq p$, a contradiction to $t \not\leqq p$ for all $t \in S$.

		 Suppose $p$ is an $S$-prime element of  $L$.  Therefore for all $a, b \in L$ with $a \cdot b \leq p$, there exists an element $s \in S$ such that $s \cdot a \leq p$ or $s \cdot b \leq p$. 
		
		Now, we claim that $(p :s)$ is a prime element of $L$. Let $c, d \in L$ such that $c \cdot d \leq  (p:s)$. This gives $(s \cdot c) \cdot d \leq  p$. Since $p$ is an $S$-prime element of $L$, we get $s^2 \cdot c  \leq  p$ or $s \cdot d \leq  p$. If $s^2 \cdot c  \leq  p$, then we get $s^3  \leq  p$ or $s \cdot c  \leq  p$. If $s^3 \leq p$, then we will get a contradiction to $t \not\leqq p$ for all $t \in S$. Therefore  $s \cdot c  \leq  p$. Thus $s \cdot c  \leq  p$  or $s \cdot d  \leq  p$, i.e., $c \leq (p:s)$ or $d \leq (p:s)$. Therefore $(p:s)$ is a prime element of $L$.

		Conversely, suppose that $(p:s)$ is a prime element of $L$ for some $s \in S$. Let $a, b \in L$ with $a \cdot b \leq p$. Since $p \leq (p : s)$. we get $a \cdot b \leq (p : s)$. As $(p : s)$ is a prime element of $L$, we get $a \leq (p :s)$ or  $b \leq (p :s)$, that is $s \cdot a \leq p$ or $s \cdot b \leq p$. This gives that $p$ is an $S$-prime element of $L$.
	\end{proof}


	\section{$S$-Ako and $S$-Oka families in $V$-lattices}
	
	In commutative ring theory, several classical results establish that certain maximal ideals are prime. To provide a unified framework for such ``maximal implies prime'' theorems, Lam and Reyes \cite{lr} introduced the concepts of \textit{Oka families} and \textit{Ako families} of ideals. Using these families, they were able to prove many classical results concerning prime ideals. 
	
	\begin{defn}[Lam–Reyes {\cite[Definitions 2.1–2.2]{lr}}]\label{def:strong-oka-ako}
		Let $R$ be a commutative ring with $1$, and let $Id(R)$ denote the set of all ideals of $R$.
		
		\begin{enumerate}
			\item An ideal family $\mathcal F\subseteq Id(R)$ with $R\in\mathcal F$ is called  a \emph{strongly Oka family} if, for all ideals $I,A\trianglelefteq R$,
			$
			(I,A)\in\mathcal F \ \text{and}\ (I:A)\in\mathcal F \ \Rightarrow\ I\in\mathcal F.$
			
			\item An ideal family $\mathcal F\subseteq Id(R)$ with $R\in\mathcal F$ is called 
			a \emph{strongly Ako family} if, for $a\in R$ and  all ideals $I, B\trianglelefteq R$,
			$(I,a)\in\mathcal F \ \text{and}\ (I,B)\in\mathcal F \ \Rightarrow\ (I,aB)\in\mathcal F.$
		\end{enumerate}
	\end{defn}

	Now, we generalized the concept of Ako family and Oka family to  $S$-Ako and $S$-Oka families in rings and then subsequently in $V$-lattices.	
	
	\begin{defn}[$S$–Ako and $S$–Oka families in rings]\label{def:S-ako-oka-ring}
		Let $R$ be a commutative ring with $1$, let $Id(R)$ denote the set of all ideals of $R$, 
		and let $S\subseteq R$ be a multiplicatively closed set. 
		Put $S^\ast=\{\,Rs:\ s\in S\,\}\subseteq Id(R)$. 
		An ideal family $\mathcal F\subseteq Id(R)$ with $S^\ast\subseteq\mathcal F$ is:
		
		\quad$\bullet$ an $S$–Ako family if, for all $I,A,B\in Id(R)$ and for all $s\in S$,
		$$
		(I,\,sA)\in\mathcal F \ \text{ and }\ (I,\,sB)\in\mathcal F \ \Longrightarrow\ (I,\,AB)\in\mathcal F;
		$$
		
		\quad$\bullet$ an $S$–Oka family if, for all $I,A\in Id(R)$ and for all $s\in S$,
		$$
		(I,\,sA)\in\mathcal F \ \text{ and }\ (I: sA)\in\mathcal F \ \Longrightarrow\ I\in\mathcal F.
		$$
	\end{defn}
	
	\noindent
	Observe that if we take $S=\{1\}$, then the notion of an $S$–Oka family  coincides with that of an Oka family. 
	Similarly, if in the definition of an $S$–Ako family we take $A=(a)$ to be a principal ideal, then the $S$–Ako family reduces to the usual Ako family.
	
	Now, let's consider the multiplicative lattice version of the $S$-Ako and $S$-Oka family.
	\begin{defn}\label{2.1.} Let $S$ be a multiplicatively closed subset of a $V$-lattice $L$.  A  family  $ \mathscr{F}$ of elements in  $L$ with $S \subseteq \mathscr{F}$ is
		said to be an \textit{$S$-Ako family} if for all $s \in S$ and for $i, a, b \in L $, ~$i \vee
		(s \cdot a), ~ i \vee (s \cdot b) \in \mathscr{F}$ implies $i \vee (a \cdot b) \in
		\mathscr{F}$.

		A family $ \mathscr{F}$ of elements in  $L$ with $S \subseteq \mathscr{F}$ is
		said to be an \textit{$S$-Oka family} if for all $s \in S$ and for $i, a \in L $, ~$i \vee
		(s \cdot a), ~ (i : (s \cdot a) )\in \mathscr{F}$ implies $i  \in
		\mathscr{F}$. Note that $(a:b) =
		\bigvee \{x \;|\;  x\cdot b \leq a\}$. Clearly, $x \cdot b \leq a
		\Leftrightarrow x \leq (a:b)$. 
		
		If the multiplicatively closed subset $S=\{1\}$, then we call $S$-Ako family and $S$-Oka family as Ako family and Oka family respectively.  
		
		Note that an Ako family is always Oka.  For this, let $\mathscr{F}$ be an Ako family in $L$ and let 	$i \vee a,~ (i:a) \in \mathscr{F}$ for $i, a \in L$. 	We need to show that $i \in \mathscr{F}$. Let $b = (i:a)$. Since $i \cdot a \leq  i$, 
		we have $i \leq (i:a) = b$, and hence 
		$i \vee b = b = (i:a) \in \mathscr{F}.$
		Since $i \vee a \in \mathscr{F}$ and $i \vee b \in \mathscr{F}$, 
		by the Ako property
		$i \vee (a \cdot b) \in \mathscr{F}.$
		Since $a \cdot b \leq i$, we have 
		$i \vee (a \cdot b) = i$. 
		Hence $i \in \mathscr{F}$.	\end{defn}
	
	First, we prove the relation between ring theoretic $S$-version of Ako and Oka family with the lattice theoretic version. For this, we need to prove the relation of lattice residuation and the {\color{black} quotient (or colon)} ideal in rings.
	
	\begin{lem}\label{lem:residual-equals-colon}
		Let $R$ be a commutative ring with $1$, and let $Id(R)$ be the multiplicative lattice of all ideals of $R$, ordered by inclusion with the multiplication as the ideal multiplication. For $A,B \in Id(R)$, define 
		$(A:B)_{\mathrm{res}} = \bigvee\{\,X \in Id(R) \mid X\cdot B \subseteq A\,\}$. 
		Then $(A:B)_{\mathrm{res}} = (A:B)$, where $(A:B) = \{\,r \in R \mid rB \subseteq A\,\}$, and equivalently, for every $X \in Id(R)$, 
		$X\cdot B \subseteq A \iff X \subseteq (A:B)$.
	\end{lem}
	
	\begin{proof}
		Clearly, $(A:B)$ is an ideal of $R$. 
		If $X \in Id(R)$ and $X B \subseteq A$, then $xB \subseteq A$ for all $x \in X$, hence $X \subseteq (A:B)$. Thus every such $X$ is contained in $(A:B)$, giving $(A:B)_{\mathrm{res}} \subseteq (A:B)$. 
		Conversely, $(A:B)B \subseteq A$ by the definition. Hence $(A:B) \leq (A:B)_{\mathrm{res}}$, yielding the reverse inclusion. Therefore $(A:B)_{\mathrm{res}} = (A:B)$, and the equivalence $X B \subseteq A \iff X \subseteq (A:B)$ follows immediately.
	\end{proof}
	
	\begin{lem}\label{thm:S-AKO-Oka-ring-vs-lattice}
		Let $R$ be a commutative ring with $1$, let $S\subseteq R$ be a multiplicatively closed set, and let $Id(R)$ as a multiplicative lattice with join $I\vee J=I+J$, product $I\cdot J=IJ$, and residual $(I:J)$. Put $S^\ast=\{\,Rs:\ s\in S\,\}\subseteq Id(R)$. Let $\mathcal F\subseteq Id(R)$ satisfy $S^\ast\subseteq\mathcal F$. Then:
		\begin{enumerate}
			\item $\mathcal F$ is an $S$–Ako family in the ring sense, i.e., for all $s\in S$ and all $I,A,B\in Id(R)$,
			\[
			(I,\,sA)\in\mathcal F \text{ and } (I,\,sB)\in\mathcal F \ \Rightarrow\ (I,\,AB)\in\mathcal F,
			\]
			if and only if $\mathcal F$ is an $S$–Ako family in $Id(R)$ in the lattice sense, i.e., for all $s\in S$ and all $i,a,b\in Id(R)$,
			\[
			i\vee (Rs\cdot a)\in\mathcal F \text{ and } i\vee (Rs\cdot b)\in\mathcal F \ \Rightarrow\ i\vee (a\cdot b)\in\mathcal F.
			\]
			\item $\mathcal F$ is an $S$–Oka family in the ring sense, i.e., for all $s\in S$ and all $I,A\in Id(R)$,
			\[
			(I,\,sA)\in\mathcal F \text{ and } (I: sA)\in\mathcal F \ \Rightarrow\ I\in\mathcal F,
			\]
			if and only if $\mathcal F$ is an $S$–Oka family in $Id(R)$ in the lattice sense, i.e., for all $s\in S$ and all $i,a\in Id(R)$,
			\[
			i\vee (Rs\cdot a)\in\mathcal F \text{ and } (i: Rs\cdot a)\in\mathcal F \ \Rightarrow\ i\in\mathcal F.
			\]
		\end{enumerate}
	\end{lem}
	
	\begin{proof}
		(1) For $i,a\in Id(R)$ and $s\in S$, one can prove that $i\vee (Rs\cdot a)=i+(Rs)a=(I,sA)$ if we put $i=I$ and $a=A$. Similarly, $i\vee (Rs\cdot b)=(I,sB)$ and $i\vee (a\cdot b)=(I,AB)$. Thus the two $S$–Ako statements are term-by-term identical.
		
		(2) As above, $i\vee (Rs\cdot a)=(I,sA)$. By the residual–colon identity in $Id(R)$, $(i:Rs\cdot a)=(I:sA)$. Hence the ring and lattice $S$–Oka statements coincide.
	\end{proof}
	
	For principal ideals and $S=\{1\}$, we get the notions of Ako family and Oka family given by Lam and Reyes \cite{lr}.
	\begin{cor}\label{cor:S-AKO-Oka-principal}
		With notation as above, for $a,b\in R$, $A=Ra$, $B=Rb$, $s\in S$, and $I\in Id(R)$:
		\[
		(I,\,sa)\in\mathcal F \text{ and } (I,\,sb)\in\mathcal F \ \Rightarrow\ (I,\,ab)\in\mathcal F,
		\]
		\[
		(I,\,sa)\in\mathcal F \text{ and } (I: sa)\in\mathcal F \ \Rightarrow\ I\in\mathcal F.
		\]
	\end{cor}

	With this preparation, we now  illustrate with examples, the concepts of $S$-Ako and $S$-Oka family in multiplicative lattices. 
	\vskip 5pt
	\begin{minipage}{.6\textwidth}

		\begin{exa}\label{4.6.} 
			Consider a lattice $L=Id(\mathbb{Z}_{12})$, the ideal lattice of the ring $\mathbb{Z}_{12} $ whose Hasse diagram is shown in Figure 2. If we take $S = \{(1), ~(4) \}$ and $\mathscr{F} = \{(1), ~ (2),~ (4)\}$, then it is easy to check that $\mathscr{F}$ is an $S$-Ako family {\color{black} as well as an $S$-Oka family} of $L$.
			
			\vskip 5pt 
			
			Now, if we take $S = \{(1)\}$ and $\mathscr{F} =  \{(1),~ (6)\}$, then $\mathscr{F}$ is not an $S$-Ako family of $L$.
			For this, put $s=(1)$, $i=(0)$ and $a=(6)=b$. Then $i \vee
			(s \cdot a), ~ i \vee (s \cdot b) \in \mathscr{F}$, but $i \vee (a \cdot b)\not \in
			\mathscr{F}$.  It is easy to check that $\mathscr{F}$ is an $S$-Oka family of $L$. Thus, $\mathscr{F}$ is an $S$-Oka family, but not an $S$-Ako family of $L$.	
		\end{exa}
	\end{minipage}\hfill
	\begin{minipage}{.3\textwidth}
		\begin{tikzpicture}[scale=.8]
			
			\draw (5,0) -- (4,1) node at (5, -0.5) {$(0)$}; \draw [fill=black] (5,0) circle
			(.1); \draw (5,0) -- (6,1) node at (6.6, 1) {$(6)$};\draw [fill=black] (6,1)
			circle (.1); \draw (6,1) -- (6,2) node at (6.6,2) {$(3)$}; \draw
			(6,1) -- (4,2); \draw [fill=black] (6,2) circle (.1); \draw (4,1) -- (4,2) node
			at (3.5, 1) {$(4)$}; \draw [fill=black] (4,1) circle (.1); \draw (4,2) -- (5,3)
			node at (3.5, 2) {$(2)$}; \draw [fill=black] (4,2) circle (.1); \draw (6,2) --
			(5,3) node at (5, 3.5) {$(1)$}; \draw [fill=black] (5,3) circle (.1); \draw node
			at (5,-1.5) {Figure 2: $L =Id( \mathbb{Z}_{12})$};
		\end{tikzpicture}
	\end{minipage}
	
	\vspace{0.5cm}

	Let $ \mathscr{F}$ be an $S$-Ako or $S$-Oka element family of a $V$-lattice $L$, then $ \mathscr{F'}$ denotes the complement of $ \mathscr{F}$ (i.e. the set consisting of all elements of $L$ that does not belong to $\mathscr{F}$).

	\textbf{\color{black}Throughout this paper, we assume that 
		if $ \mathscr{F}$ is an $S$-Ako or an $S$-Oka family in a $V$-lattice $L$ with 1 compact, then
		$ \mathscr{F'}$ has a maximal element in $L$, i.e., $Max(
		\mathscr{F'})  \not = \emptyset$.
	}

	\begin{thm}[$S$–Prime Element Principle ($S$–PEP)]\label{2.3.}
		Let $L$ be a $V$–lattice with 1 compact, and let $S$ be a multiplicatively closed subset of $L$. 
		If $\mathscr{F}$ is an $S$–Ako family or an $S$–Oka family in $L$, and if for every $m \in Max(\mathscr{F}')$,  $t\nleq m$ for all $t\in S$, 
		then $\mathscr{F}'$ is an \emph{$MSP$–family}; that is,
		$Max(\mathscr{F}') \subseteq Spec_S(L),$
		where $Spec_S(L)$ denotes the set of all $S$–prime elements of $L$.
	\end{thm}
	\begin{proof} Let  $i \in Max( \mathscr{F'})$. Suppose $i$ is  not an $S$-prime element of $L$. Then for all $s \in S$, there exist $a, b \in L$ such
		that $a \cdot b \leq i$ with $s \cdot a \nleq i$ and $s \cdot b \nleq i$. This gives $i < (i \vee
		(s \cdot a)), ~i < (i \vee
		(s \cdot b))$.  Since $a \cdot b \leq i$, we have $(s\cdot a) \cdot (s \cdot b) \leq i$, equivalently, we have $s \cdot b \leq i: (s\cdot a)$. Further, if $i = (i:(s \cdot a))$, then   $s \cdot b \leq i$, a contradiction. Hence  $i < (i:(s \cdot a))$. Since $i$ is a maximal element of $\mathscr{F'}$, we have  $(i \vee (s \cdot a))$, $(i \vee (s \cdot b))$
		and $(i:(s \cdot a))$ all belong to $ \mathscr{F}$. However, $i = (i \vee (a
		\cdot b)) \not\in \mathscr{F}$, a contradiction to our assumption that $ \mathscr{F}$ is a 
		$S$-Ako family or $S$-Oka family.\end{proof}
	
	\begin{rem}\label{2.4.} The converse of Theorem \ref{2.3.} need not be true
		in general. Consider the multiplicative lattice of ideals of $< \mathbb{Z}_{12},
		+, \cdot
		>$ with the multiplication as the ideal multiplication.

		Consider,  $S = \{(1)\}$ and $ \mathscr{F}=\{(4), (6), (2), (1)\}$. Then $
		\mathscr{F^{'}}= \{ (0), (3)\}$ is a $MSP$-family (since
		$Max(\mathscr{F^{'}})= \{ (3)\} \subseteq Spec(L)$). But $
		\mathscr{F}$ is not an $S$-Oka family, since $((0) \vee (6)), ((0):(6))=(2) \in
		\mathscr{F}$ but $(0) \not\in \mathscr{F}$. Also, $\mathscr{F}$ is
		not a  $S$-Ako family, since $((0) \vee (6)), ((0) \vee(6)) \in \mathscr{F}$ but
		$((0) \vee ((6) \cdot (6)))=(0) \not\in \mathscr{F}$.
		
	\end{rem}

	\begin{defn}\label{filter}
		Let $\mathscr{F}$ be a family of elements with $1 \in \mathscr{F}$ in a $V$-lattice $L$ with 1 compact, then
		
		\begin{enumerate} 
			\item $ \mathscr{F}$ is a \textit{semi-filter} if for all $i , j \in L$, $j \leq i$ and $j \in \mathscr{F}$
			implies $i \in \mathscr{F}$, \item $ \mathscr{F}$ is a \textit{filter} if $
			\mathscr{F}$ is a semi-filter and if $i, j \in \mathscr{F}$
			implies $i \wedge j \in \mathscr{F}$, \item $ \mathscr{F}$ is
			\textit{$M$-closed} if $i, j \in \mathscr{F}$ implies $i
			\cdot j \in \mathscr{F}$.
		\end{enumerate}
	\end{defn}

	\begin{lem}\label{2.6.}  Let $S$ be a multiplicatively closed subset of a $V$-lattice $L$ with 1 compact. Let $  \mathscr{F}$  be a family of elements  of $L$
		with $S \subseteq \mathscr{F}$.  If $\mathscr{F}$ is a semi-filter family  satisfying $M$-closed property {\color{black}(or a filter family with $M$-closed property)}, then {\color{black} $\mathscr{F}$ is an $S$-Ako and an $S$-Oka family in $L$.}
	\end{lem}
	
	\begin{proof}  {\color{black} It is easy to prove that $ \mathscr{F}$ is a semi-filter family satisfying $M$-closed property if and only if $ \mathscr{F}$ is  a filter family  satisfying $M$-closed property. Let $ \mathscr{F}$ be a semi-filter family  satisfying $M$-closed property.} We show that  $ \mathscr{F}$ is an $S$-Ako family.  Let $i \vee
		(s \cdot a), ~ i \vee (s \cdot b) \in \mathscr{F}$ for  $i, a, b \in L $ and for all $s \in S$. Since $ \mathscr{F}$ satisfies the $M$-closed property, we get $(i \vee
		(s \cdot a)) \cdot (i \vee (s \cdot b)) \in \mathscr{F}$. We have  $(i \vee
		(s \cdot a)) \cdot (i \vee (s \cdot b)) \leq   i \vee (a \cdot b)$. By the semi-filterness property of $\mathscr{F}$, we have $ i \vee (a \cdot b)\in
		\mathscr{F}$.  Hence $\mathscr{F}$ is an $S$-Ako  family in $L$.
		
		{\color{black} Now, we show that $\mathscr{F}$ is an $S$-Oka family in $L$. Let $i, a, b \in L $ such that $i\vee sa$, $(i:sa)\in \mathscr{F}$ for all $s\in S$. Then $i\vee a, (i:a)\in \mathscr{F}$. Since $ \mathscr{F}$ satisfies the $M$-closed property, we have $(i\vee a)(i:a)\in \mathscr{F}$. Now, $(i\vee a)(i:a)=i(i:a)\vee a(i:a)\leq i\vee i=i$. As $\mathscr{F}$ is semi-filter, we get $i\in \mathscr{F}$. Thus, $\mathscr{F}$ is an $S$-Oka family in $L$.}	
	\end{proof}

	\begin{lem}\label{2.11.}  Let $S$ be a multiplicatively closed subset of a $V$-lattice $L$ with 1 compact.  Let $\{p_{i}\}$ be a fixed subset of $Spec_s(L)$. Then the family
		$ \mathscr{F}= \{j \in L \;|\; (s \cdot j) \nleq p_{i}$, for all $p_i$ and for all $s \in S$ \}.
		Then $ \mathscr{F}$ is an $S$-Ako family of $L$.\end{lem}
	
	\begin{proof} Let $\{p_{i}\}$ be a fixed subset of $Spec_s(L)$. Since each $p_i$ is an $S$-prime element of $L$, we have $t \nleq p_i$ for all $t \in S$ and for all $p_i$. Hence $S \subseteq \mathscr{F}= \{j \in L \;|\; (s \cdot j) \nleq p_{i}$, for all $p_i$ and for all $s \in S$\}. Let $j_{1}, j_{2} \in \mathscr{F}$. Therefore $(s \cdot j_{1}) \nleq p_{i}$ for all $p_i$ and for all $s \in S$,
		$(s \cdot j_{2}) \nleq p_{i}$ for all $p_i$ and for all $s \in S$. We claim that $ j_1 \cdot j_2 \in \mathscr{F}$. Suppose on the contrary that $j_1 \cdot j_2 \not\in \mathscr{F}$. Therefore there exists an element $s_1 \in S$ and $p_r \in Spec_s(L)$ such that $s_1 \cdot j_1 \cdot j_2 \leq p_r$. Since $p_{r}$ is an $S$-prime
		element, there exists an element $s_2 \in S$ such that  $s_2 \cdot s_1 \cdot j_{1} \leq p_r$ or $ s_2 \cdot j_{2} \leq p_{r}$, a contradiction to $j_1, j_2 \in  \mathscr{F}$. Hence $ j_1 \cdot j_2 \in \mathscr{F}$. This gives 
		$ \mathscr{F}$ is a $M$-closed subset of $L$.

		{\color{black} Now, we prove that $\mathscr{F}$ is a semi-filter. Let $j\leq i$ for $i,j\in L$ and  $j\in \mathscr{F}$. Then $ (s \cdot j) \nleq p_{i}$, for all $p_i$ and for all $s \in S$. If $ (s \cdot i) \leq p_{i}$, for some $p_i$ or for some $s \in S$. Then $ (s \cdot j)\leq (s \cdot i) \leq p_{i}$, for some $p_i$ or for some $s \in S$, which is a contradiction. Therefore, $ (s \cdot i) \nleq p_{i}$, for all $p_i$ and for all $s \in S$. Hence, $\mathscr{F}$ is a semi-filter.
			
			Thus, by Lemma \ref{2.6.}, $ \mathscr{F}$ is an $S$-Ako family of $L$.}
		
	\end{proof}

	\section{Applications of $S$-Prime Element Principle}

	The \textit{Prime Ideal Principle} provides a unifying framework that guarantees the existence of prime ideals under suitable conditions.  The strength of this principle lies in its abstract nature—it applies not only to rings but also to lattices, and other algebraic systems where the notion of primality can be suitably defined. 
	In the context of multiplicative lattices and related ordered structures, it provides a powerful tool to establish the existence  of prime elements  for further structural investigations.
	
	In this section, we introduces the $S$-Prime Ideal Principle in a general setting and explores its consequences in various contexts.
	
	\begin{lem}\label{5.1.}  Let $S$ be a multiplicatively closed subset of a $V$-lattice $L$ with 1 compact.  Then \linebreak $\mathscr{F}=\{i \in L \;|\; s \leq i$ for some $s \in S \}$ is an $S$-Ako family in $L$.\end{lem}
	
	\begin{proof} Let $(a \vee (s \cdot b)), (a \vee (s \cdot c)) \in \mathscr{F}$ for $a, b , c \in L$ and for all $s \in S$. Then $s_{1} \leq (a \vee (s \cdot b))$ for some $s_{1} \in S$ and $s_{2} \leq (a \vee (s \cdot c))$ for some $s_{2} \in S$. This gives
		$ (s_{1} \cdot s_{2}) \leq ((a \vee (s \cdot b)) \cdot  (a \vee (s \cdot c))) \leq (a
		\vee (b \cdot c))$. Therefore, $a\vee bc\in \mathscr{F}$. Thus, $ \mathscr{F}$ is an $S$-Ako family. \end{proof}
	
	\begin{prop}[Lam–Reyes Proposition  3.1 \cite{lr}]
		Let $S$ be a  multiplicative subset of a commutative ring $R$ with 1. Let 
		$\mathcal F=\{\,I\lhd R \mid I\cap S\not =\varnothing\,\}$. 
		Then $\mathcal F$ is  strongly Ako, and 
		$Max(\mathcal F')\subseteq Spec(R)$.
	\end{prop}
	
	\begin{proof} Consider the  multiplicative $V$–lattice $L=Id(R)$ (with $1$ compact) and put 
		$S^\ast=\{\,Rs \mid s\in S\,\}\subseteq L$. 
		By Lemma~\ref{5.1.}, the family $\mathscr G \;=\; \{\, J\in L \mid \exists s\in S: Rs\le J \,\}
		\;=\; \{\, J\in Id(R) \mid J\cap S\neq\varnothing \,\}
		$
		is an $S^\ast$–Ako family in $L$ which is also  the ideal family $\mathcal F$ in the statement. Further, for any $m \in Max(\mathcal F')$, we have $t \nleq m$ for all $t \in S^*$. By Theorem \ref{2.3.},  $Max(\mathcal F')\subseteq Spec(R)$.
	\end{proof}

	Let $S=\{1\}$ be a multiplicatively closed subset of a $c$-lattice $L$. 
	If $\mathscr{F}=\{1\}$, then $\mathscr{F}$ is an $S$–Ako family. 
	In this case, the $S$–Prime Element Principle yields the standard conclusion 
	that $Max(L)\subseteq Spec(L)$; that is, every maximal element of a $c$-lattice is prime. 
	It is well known that an ideal $I$ is a maximal (respectively, prime) ideal of a commutative ring $R$ with $1$ 
	if and only if $I$ is a maximal (respectively, prime) element of the multiplicative lattice $Id(R)$. 
	Combining these facts, we obtain that in any commutative ring $R$ with $1$, every maximal ideal is prime.

	\begin{defn} [\cite{ AAJ, MC1, SJ}] 	An element $a $ of a multiplicative lattice $L$ is {\it nilpotent} if $a^{n}= 0$ for some $ n \in
		\mathbb{Z}^{+}$. A multiplicative lattice $L$ is said to be {\it reduced}, if zero is the only nilpotent element of $L$. Let 
		$a$ be any element of a multiplicative lattice $L$, we define $a^*=\bigvee\{x \in L  \;|\: a^n \cdot x=0$ for some $ n \in \mathbb{Z}^{+} \}$.
		If $L$ is reduced, then $a^*=\bigvee\{x \in L \;|\: a \cdot
		x=0\}$.  An element $j$ of multiplicative lattice is called \textit{dense} element if $(0:j) = 0$.
	\end{defn}

	\begin{lem}\label{5.7.} 
		The family $\mathscr{F}=\{\,j\in L \mid j^{*} = 0\,\}$ is an $S$–Ako family in $L$, where $S$ is a multiplicatively closed subset of a $c$–lattice $L$ such that $S\subseteq \mathscr{F}$.
	\end{lem}
	
	\begin{proof}
		
		Let $a,b\in \mathscr{F}$. Then $a^{\star}=0$ and $b^{\star}=0$. Let $x\in L$ such that $(ab)^nx=0$. Then $a^n(b^nx)=0$. Since $a^{\star}=0$, we have $b^nx=0$. Also $b^{\star}=0$ gives $x=0$. This proves $(ab)^{\star}=0$. Hence, $ab\in \mathscr{F}$. We proved that $\mathscr{F}$ is $M$-closed set. Now, we prove that $\mathscr{F}$ is a semi-filter. Let $a\leq b$ and $a\in \mathscr{F}$. Then $a^{\star}=0$. Let $z\in L$ such that $b^nz=0$. Then $a^nz\leq b^nz=0$. This gives $a^nz=0$. Since $a^{\star}=0$, we have $z=0$. Hence, $b^{\star}=0$. Therefore, $b\in \mathscr{F}$. Therefore, $\mathscr{F}$ is a semi-filter. Thus, By Lemma $\ref{2.6.}$, $\mathscr{F}$ is an $S$-Ako family.
	\end{proof}

	\begin{lem}\label{15.7.} The family $ \mathscr{F}= \{j \in L \;|\; (0:j)={0}\}$ (the set of all dense elements of $L$) is an $S$-Ako family in $L$, where $S$ is a multiplicatively closed subset of a $c$-lattice $L$ such that $S\subseteq \mathscr{F}$.
	\end{lem}
	\begin{proof}
		Let $a,b\in \mathscr{F}$. Then $(0:a)=0$ and $(0:b)=0$. Let $x\in L$ such that $(ab)x=0$. Then $a(bx)=0$. This gives $bx\leq (0:a)=0$. Therefore, $bx=0$. This implies $x\leq (0:b)=0$. Hence, $x=0$. This proves $(0:ab)=0$. That is $ab\in \mathscr{F}$. We proved that $\mathscr{F}$ is $M$-closed set. Now, we prove that $\mathscr{F}$ is semi-filter. Let $a\leq b$ and $a\in \mathscr{F}$. Therefore, $(0:a)=0$. Let $z\in L$ such that $bz=0$. Then $az=0$. This gives $z=0$. Hence, $(0:b)=0$. Therefore, $b\in \mathscr{F}$. By Lemma $\ref{2.6.}$, $\mathscr{F}$ is an $S$-Ako family.
	\end{proof}

	\begin{thm}\label{5.8.} Let $L$ be a $c$-lattice, an element maximal with respect to not being dense element is  prime.\end{thm}
	
	\begin{proof}  follows from   Lemma \ref{15.7.} and $S$-Prime Element Principle.\end{proof}
	
	As an immediate consequence of the above result is the following corollary.
	
	\begin{cor}
		Let $R$ be a commutative ring with 1. An ideal maximal with respect to not being dense is prime.
	\end{cor}

	\begin{thm}\label{15.8.} Let $L$ be a $c$-lattice, an element maximal with respect to being zero-divisors of $L$ is a prime element of $L$.\end{thm}
	
	\begin{proof} Let $Z(L)=\{x\in L\;|\; x\cdot y=0 \text{ for some nonzero } y \in L\}$ be the set of all zero-divisors. Let $\mathscr{F}=L\setminus Z(L)$ be the set of all non zero-divisors of $L$. Clearly, $x \in \mathscr{F}$ if and only if  $(0:x) = 0$. Therefore $\mathscr{F}$ is the set of all dense elements of $L$. Hence proof follows from   Theorem \ref{5.8.}.\end{proof}
	
	\begin{defn}\label{1.65.} 
		
		A proper element $a$ (i.e., $a\not=1$) of a  multiplicative lattice $L$ is an \textit{annihilator element}, if $a = (0:x)$ for some nonzero $x \in L$.	
		
	\end{defn}
	
	\begin{rem}
		Let $S = \{1\}$, a multiplicatively closed subset  of a $c$-lattice $L$. Then   the set $\mathscr{F}$ of all non annihilator elements of $L$ need not be $S$-Ako family of $L$. \end{rem}
	\begin{minipage}{.6\textwidth}
		\begin{exa}
			Consider a lattice $K$ whose Hasse diagram is shown in Figure
			3. On $K$, define the trivial multiplication $x \cdot y = 0 = y \cdot
			x$ for every $x, ~y \not\in \{1\} $ and $x \cdot 1 = x = 1 \cdot x$ for
			every $x \in K$. It is easy to see that $K$ is a multiplicative
			lattice. If we take $S = \{1\}$ and  $\mathscr{F} = \{a,~ b,~c, ~1\}$, set of all non annihilator elements of $K$, then $\mathscr{F}$ is not an $S$-Ako in $K$. As $(0 \vee a), (0 \vee b) \in \mathscr{F}$, but $0 \vee (a \cdot b)  =  0  \not\in \mathscr{F} $. Hence $\mathscr{F}$ is not an $S$-Ako family. 
		\end{exa}
	\end{minipage}	\hfill
	\begin{minipage}{.4\textwidth}
		\begin{center}
			\begin{tikzpicture}[scale = .81]
				\draw (5,0) -- (4,1) node at (5, -0.4) {$0$}; \draw [fill=black] (5,0) circle
				(.1); \draw (5,0) -- (6,2) node at (6, 2.3) {$c$}; \draw [fill=black] (6,2)
				circle (.1); \draw (4,1) -- (4,2) node at (3.7, 1) {$a$}; \draw [fill=black]
				(4,1) circle (.1); \draw (4,2) -- (5,3) node at (3.7, 2) {$b$};
				\draw [fill=black] (4,2) circle (.1); \draw (6,2) -- (5,3) node at (5.3, 3)
				{$d$}; \draw [fill=black] (5,3) circle (.1); \draw (5,3) -- (5,4) node at
				(5,4.3) {$1$}; \draw [fill=black] (5,4) circle (.1); \draw node at (5,-1)
				{Figure 3: A multiplicative lattice $K$};
			\end{tikzpicture}
		\end{center}
	\end{minipage}

	\begin{thm}\label{anni}
		Let $S = \{1\}$, a multiplicatively closed subset  of a $c$-lattice $L$. Then an element  maximal with respect to being annihilator element of $L$ is a prime element of $L$.
	\end{thm}

	\begin{proof}
		Let $\mathscr{F}$ be a set of all non annihilator elements of $L$. Since $1$ is not an annihilator element, $S \subseteq \mathscr{F}$. We claim that $\mathscr{F}$ is an $S$-Oka family. Let $(i \vee j), (i : j) \in \mathscr{F}$. Therefore $(i \vee j) \not=  (0 : x)$ and $ (i : j) \not=  (0 : x)$ for all nonzero $x \in L$. We claim that $i \in \mathscr{F} $. Suppose on the contrary that $i \not \in \mathscr{F} $. So $i = ( 0 : x)$ for some nonzero $x \in L$. Thus, $(i : j) = ((0 : x) : j) = ( 0 : (x \cdot j))$. Note that $(x \cdot j) \not = 0$. If $(x \cdot j)  = 0$, then $j \leq (0 : x) = i$. This gives $i \vee j = i = ( 0: x)$ for nonzero $x \in L$, a contradiction to $(i \vee j) \in \mathscr{F}$. So  $(i : j) = ( 0 : (x \cdot j))$, where $(x \cdot j) \not = 0$, a contradiction to $(i : j ) \in \mathscr{F}$. Hence $\mathscr{F} $ is an $S$-Oka family of $L$. By Theorem \ref{2.3.}, $ \mathscr{F'}$ is a \textit{$MSP$-family}, that is, $Max(
		\mathscr{F'}) \subseteq Spec(L)$, i.e., an element  maximal with respect to being annihilator element of $L$ is  prime. \end{proof}
	As a consequence of Theorem \ref{anni}, we have:
	
	\begin{cor}
		Let $R$ be a commutative ring with 1. Then an ideal  maximal with respect to being annihilator ideal of $R$ is a prime ideal of $R$.
	\end{cor}	
	
	\begin{defn}[{G. C\u{a}lug\u{a}reanu \cite{C}}] An element  $a$ of a  lattice $L$ is called an \textit{essential} element if there is no nonzero $x \in L $ such that  $a \wedge x  = 0$. \end{defn}
	
	\begin{thm}
		Let $S = \{1\}$, a multiplicatively closed subset  of a reduced $c$-lattice $L$. Then an element  maximal with respect to not being essential element of $L$ is a prime element of $L$.
	\end{thm}

	\begin{proof}
		Let $\mathscr{F}$ be a set of all  essential elements of $L$. Since $1$ is a essential element of $L$, $S \subseteq \mathscr{F}$. We claim that $\mathscr{F}$ is a {\color{black} semifilter} that satisfies $M$-closed property. Let $i, j \in \mathscr{F}$ and suppose that $i \cdot j \not \in \mathscr{F}$. So there exists a non-zero element $k \in L$ such that $((i \cdot j) \wedge k) = 0$. Since  $(i \cdot j \cdot k) \leq ((i \cdot j) \wedge k) = 0$ implies $(i \cdot j \cdot k) = 0$. Since $L$ is reduced, we have $i \wedge (j \cdot k)=0$. 
		This together with $i$ is an essential element of $L$, we get $j \cdot k = j \wedge k =0$. Again, $j$ is an essential element of $L$, we get $k = 0$, a contradiction to $k$ is nonzero element of $L$. Hence $\mathscr{F}$ is an $M$-closed subset of $L$. Now, we show that if $x \in \mathscr{F}$ and if $y$ is any element of $L$ such that $x
		\leq y$, then $y \in \mathscr{F}$. Suppose on the contrary that $y$ is not
		in $\mathscr{F}$. Then there exists $z \neq 0$ such that $ (y \cdot
		z)=0$. But $x \leq y$ gives
		$(x \cdot z) \leq (y \cdot z)=0$, a
		contradiction to $x \in \mathscr{F}$. Hence $y \in \mathscr{F}$. Therefore $\mathscr{F}$ is a semifilter with $M$-closed property. By Lemma \ref{2.6.},  $ \mathscr{F}$ is an $S$-Ako and an $S$-Oka family in a reduced $c$-lattice $L$. By Theorem \ref{2.3.}, $ \mathscr{F'}$ is a \textit{$MSP$-family}, that is, $Max(
		\mathscr{F'}) \subseteq Spec(L)$, that is, an element maximal with respect to not being essential element is a prime element of $L$. 
		
	\end{proof}

	Now, we introduce the $S$-Prime Element Principle Supplement (S-PEPS). This Principle enables
	one to observe that the behavior of an $S$-prime element influences the
	behavior of all elements in a $V$-lattice.
	
	\begin{thm}\label{2.5.} $S$-Prime Element Principle Supplement (S-PEPS) : Let $S$ be a multiplicatively closed subset of a $V$-lattice $L$ with 1 compact. Let $ \mathscr{F}$ be an $S$-Ako or an $S$-Oka family in 
		$L$. Assume that every nonempty chain of elements in $
		\mathscr{F'}$ (with respect to $\leq$) has an upper bound in $
		\mathscr{F'}$ and for $m \in Max(
		\mathscr{F'}) $, we have $t \nleq m$ for all $t \in S$.
		
		\begin{enumerate} \item  Let $ \overline{\mathscr{F}}$ be a semi-filter of elements in $L$. If every $S$-prime
			element in $ \overline{\mathscr{F}}$ belong to $ \mathscr{F}$,
			then $ \overline{\mathscr{F}} \subseteq  \mathscr{F}$. \item Let
			$j \in L$. If all $S$-prime elements containing $j$ are in $
			\mathscr{F}$, then all elements containing $j$ are in $
			\mathscr{F}$. \item If all the $S$-prime elements belong to $
			\mathscr{F}$, then all elements of $L$ belong to $ \mathscr{F}$.
		\end{enumerate}
	\end{thm}
	
	\begin{proof}\begin{enumerate} \item Suppose $ \overline{\mathscr{F}} \not \subseteq  \mathscr{F}$. So there exists an element $i \in  \overline{\mathscr{F}}$ such that $ i \notin \mathscr{F}$. By
			the hypothesis on $ \mathscr{F'}$ (together with Zorn's Lemma), $i
			\leq p$ for some $p \in Max( \mathscr{F'})$. Since $i \in  \overline{\mathscr{F}}$ with $i \leq p$ and $
			\overline{\mathscr{F}}$ is a semi-filter, we get $ p \in
			\overline{\mathscr{F}}$. As $p \in Max( \mathscr{F'})$, by Theorem \ref{2.3.}, $p$ is an $S$-prime element of $L$. We have assumed that
			every $S$-prime
			element in $ \overline{\mathscr{F}}$ belong to $ \mathscr{F}$, so $ p \in \mathscr{F}$, a contradiction to
			$p \in Max( \mathscr{F'})$. Hence $ \overline{\mathscr{F}} \subseteq  \mathscr{F}$.
			
			\item follows from $(1)$ by
			taking $ \overline{\mathscr{F}}$ to be the semi-filter consisting
			of all elements containing $j$.
			\item follows from
			$(1)$ by taking $ \overline{\mathscr{F}}$ to be the family of all
			elements in $L$.
	\end{enumerate} \end{proof}
	
	One of the main result of this section is to prove an analogue of Cohen's Theorem which states that if all the prime ideals are finitely generated, then all ideals are finitely generated, that is, the ring is Noetherian. To prove this result, we introduce the $S_{pr}$-Oka family. A family $ \mathscr{F}$ of elements in  $L$ with $S~{\color{black}\subseteq Pr(L)} \subseteq \mathscr{F}$ is
	said to be an \textit{$S_{pr}$-Oka family} if for all $s \in S$ and for $i\in L$, and $a \in Pr(L) $, ~$i \vee
	(s \cdot a), ~ (i : (s \cdot a) )\in \mathscr{F}$ implies $i  \in
	\mathscr{F}$, where $Pr(L)$ is the set of all principal elements of $L$. By \textit{principal element}, we mean an element which is both meet principal as well as join principal.  An element $m$ is said to be\textit{ meet principal} if $a \wedge mb = m((a:m) \wedge b)$ for all $a,b \in L$ and $m$ is said to be \textit{join principal} if $a \vee (b :m) = (am \vee b): m$ for all $a,b \in L$. It is well known that the product of meet (join) principal element is again a meet (join) principal element, see Dilworth \cite[Lemma 3.3]{D}. Hence $Pr(L)$ is a multiplicatively closed subset of $L$. In fact, every principal element is compact. Hence $Pr(L)\subseteq L_*$.

	A multiplicaive lattice $L$ is said to be an \textit{$r$-lattice}, if  $L$ is a modular, principally generated (every element is a join of principal elements), and compactly generated lattice with $1$ as the compact element of $L$; see D. D. Anderson \cite{A}. Thus in an r-lattice, the product of two compact elements is compact. A natural example of an $r$-lattice is the lattice of all ideals of a commutative ring $R$ with 1.

	In $r$-lattice $L$, to prove a proper element $p$ is prime, it is enough to prove that if $a\cdot b \leq p$ implies that $a \leq p$ or $b \leq p$ for $a, b \in Pr(L)$. 
	
	An analogue of Theorem \ref{2.3.} is true for $S_{pr}$-Oka family in $r$-lattices. For the sake of completeness, we explicitly, quote the result.
	
	\begin{thm}[$S_{pr}$–Prime Element Principle ($S_{pr}$–PEP)]\label{spr}
		Let $L$ be a $r$–lattice and let $S ~{\color{black}\subseteq Pr(L)}$ be a multiplicatively closed subset of $L$. 
		If $\mathscr{F}$ is an  $S_{pr}$–Oka family in $L$, and if for every $m \in Max(\mathscr{F}')$,  $t\nleq m$ for all $t\in S$, 
		then $\mathscr{F}'$ is an \emph{$MSP$–family}; that is,
		$Max(\mathscr{F}') \subseteq Spec_S(L),$
		where $Spec_S(L)$ denotes the set of all $S$–prime elements of $L$.
	\end{thm}
	\begin{proof}
		Follows on similar lines of Theorem \ref{2.3.}.
	\end{proof}
	
	With this preparation, we are now ready to prove that the set of all compact elements is an $S_{pr}$-Oka family.
	
	\begin{thm}\label{spr-oka}
		Let $L$ be an $r$–lattice and let $L_{*}$ denote the set of all compact elements of $L$. 
		Then $L_{*}$ is an $S_{\mathrm{pr}}$–Oka family in $L$; that is, for every $s\in S~{\color{black}\subseteq Pr(L)}$ and every principal element $a\in L$, 
		\[
		i\vee (s\cdot a)\in L_{*}\ \text{and}\ (i:s\cdot a)\in L_{*}\quad\Longrightarrow\quad i\in L_{*}.
		\]
		In particular, if both $i\vee  a$ and $(i:a)$ are compact for  principal element $a\in L$, then $i$ is compact.
	\end{thm}
	
	\begin{proof}
		Since $L$ is an $r$-lattice, it is compactly generated. Hence the elements $i \vee (s\cdot a)$ and $(i: (s\cdot a))$ are a join of compact elements. However $i\vee (s\cdot a)\in L_{*}\ \text{and}\ (i:(s\cdot a))\in L_{*}$, we can write $(i:(s\cdot a))=\bigvee_{k=1}^n p_k$ and $i\vee (s\cdot a)=(s\cdot a)\vee \bigvee_{j=1}^m q_j$, 
		where each $p_k,q_j\in L_{*}$ and $q_j\le i$.  
		Set $P:=\bigvee_{k=1}^n p_k$ and $Q:=\bigvee_{j=1}^m q_j$; then $Q\le i$ and $i\vee (s\cdot a)=(s\cdot a)\vee Q$.
		
		Since $L$ is modular and $Q\le i$, we have
		$ i = \big((s\cdot a)\vee Q\big)\wedge i = Q\vee\big((s\cdot a)\wedge i\big).$
		By Corollary 3.3 of Dilworth \cite{D}, $s\cdot a$ is principal and hence compact.  Hence $
		(s\cdot a)\wedge i = (s\cdot a)\cdot (i:(s\cdot a)) = (s\cdot a)\cdot P$. 
		Thus, $i = Q\vee\big((s\cdot a)\cdot P\big)
		= (\bigvee_{j=1}^{m} q_j) \ \vee\  (\bigvee_{k=1}^{n} \big((s\cdot a)\cdot p_k\big))$,  as 
		 multiplication distributes over  joins.
		
		Since each $q_j$ and $p_k$ is compact, 
		and $(s\cdot a)\cdot p_k$ is compact, we have $i$ is the finite join of compact elements, so $i$ is compact. Thus $L_{*}$ satisfies the $S_{\mathrm{pr}}$–Oka property.
	\end{proof}
	
	As a consequence of Theorem \ref{spr-oka}, we have:
	
	\begin{cor}
		Let $L$ be an $r$-lattice. Then every element maximal with respect to not being compact is prime.
	\end{cor}
	
	\begin{proof}
		Follows from Theorem \ref{spr} and Theorem \ref{spr-oka}. 
	\end{proof}
	It is well known that the compact elements of $Id(R)$ are nothing but the finitely generated ideals of a commutative ring $R$ with 1. Hence, we have the following corollary.
	
	\begin{cor}[Cohen \cite{cohen}]
		Let $I$ be an ideal in a commutative  ring $R$ with 1 that is maximal with respect to not being finitely generated. Then $I$ is prime.
	\end{cor}

	\begin{thm}
		Let $S = \{1\}$, a multiplicatively closed subset  of a $r$-lattice $L$. Then $L$ is Noetherian  if and only if  every prime element is compact.
	\end{thm}
	\begin{proof}
		Let $\mathscr{F}$ be a set of all compact elements of $L$. Since $1$ is a principal and hence compact element of $L$, $S~{\color{black}\subseteq Pr(L)}  \subseteq \mathscr{F}$. Then $\mathscr{F}$ is an $S_{pr}$-Oka family follows from Theorem \ref{spr-oka}. If all prime elements are compact, then by $S$-Prime Element Principle Supplement, all elements of $L$ are compact, i.e., $L$ is a Noetherian lattice. 
	\end{proof}
	
	\begin{lem} \label{1.67.}
		Let $S$ be a multiplicatively closed subset  of a $c$-lattice $L$ and $\mathscr{F}$ be a set of  elements of $L$ such that $S \subseteq \mathscr{F}$ with $i, ~j$ be any elements of $L$. Then the following statements are equivalent:
		
		\begin{enumerate}
			\item \label{1.67.1}$i \leq sj$ and $sj, ~(i:sj) \in \mathscr{F}$ for all $s \in S \Rightarrow i \in \mathscr{F}$.
			\item \label{1.67.2} $\mathscr{F}$ is an $S$-Oka family. 
		\end{enumerate}	
	\end{lem}
	
	{\color{black} \begin{proof}
			(\ref{1.67.1}) $\implies$ (\ref{1.67.2}) Let $i\vee sj, (i:sj)\in \mathscr{F}$. We claim that $i\in \mathscr{F}$. Clearly, $i\leq (i\vee sj)$ and $(i\vee sj)\in \mathscr{F}$. Also, $i:(i\vee sj)=(i:i)\wedge (i:sj)=1\wedge (i:sj)=(i:sj)\in \mathscr{F}$. Therefore, by (\ref{1.67.1}), $i\in \mathscr{F}$. This proves that $\mathscr{F}$ is an $S$-Oka family.
			
			(\ref{1.67.2}) $\implies$ (\ref{1.67.1}) Obvious.
	\end{proof}}

	We found  following result in \cite{D}. 
	
	\begin{lem}[Dilworth \cite{D}] \label{1.68.}
		If $M_1$ and $M_2$ are meet principal elements, then $M_1\cdot M_2$ is a \mbox{meet principal.}
	\end{lem}
	
	\begin{lem}\label{1.69.}
		If $i\leq j$ and $j$ is meet principal element in a $c$-lattice $L$, then $j \cdot (i:j) = i$.
	\end{lem}
	\begin{proof}
		Since $(i:j)\leq (i:j)$, we always have $j \cdot (i:j) \leq i$. Since $j$ is meet principal, by Corollary 3.1 in \cite{D}, we have $(b:j)j=b\wedge j$ for all $b\in L$. Take $b=i$, we get $(i:j)j=i\wedge j=i$.
	\end{proof}

	\begin{thm} \label{1.71.}
		Let $S = \{1\}$ be a multiplicatively closed subset  of a $c$-lattice $L$. Then 
		every element of $L$ is meet principal if and only if every prime element of $L$ is meet principal.

	\end{thm}
	\begin{proof}
		Let $\mathscr{F}$ be set of all set of all meet principal elements of $L$. Since $L$ is a $c$-lattice, By Lemma \ref{1.68.}, $\mathscr{F}$ is an $M$-closed set of $L$. Let $i, ~j$ be any elements of $L$ such that $i \leq j$ with $j , ~ (i:j) \in \mathscr{F}$. Then by Lemma \ref{1.69.}, $j \cdot (i:j) = i$. Since $\mathscr{F}$ is an $M$-closed set, we get $i \in \mathscr{F}$. Therefore, by Lemma \ref{1.67.}, $\mathscr{F}$ is an $S$-Oka family.  If all prime elements are meet principal, then by $S$-Prime Element Principle Supplement, all elements of $L$ are meet principal.
	\end{proof}

	\begin{thm}\label{2.60.}
		Let $S = \{1\}$ be a multiplicatively closed subset of a $c$-lattice $L$ and let $i_{\alpha} $, $\alpha \in \Lambda$ ($\Lambda$ is an indexed set) be the fixed set of elements of $L$. Then any element maximal with respect to not containing a finite product of  $i_{\alpha}$'s is prime element. If the   $i_{\alpha}$'s are compact element and every prime element of $L$ contains some  $i_{\alpha}$, then a finite product of some of  $i_{\alpha}$ is zero element of $L$.
	\end{thm}
	
	\begin{proof}
		Let $\mathscr{F}$ be set of all elements  that contains a (finite) product of the  $i_{\alpha}$'s.  Clearly, $\mathscr{F}$ is filter family with satisfying $M$-closed property. By Lemma \ref{2.6.}, $\mathscr{F}$ is an $S$-Ako family. and by  Theorem \ref{2.3.} i.e. $S$-Prime Element Principle, we get any element maximal with respect to not containing a finite product of  $i_{\alpha}$'s is prime element. Now, suppose that  $i_{\alpha}$'s are compact elements. Let $C$ be chain of elements of $\mathscr{F}'$. We claim that $C$ has an upper bound in $\mathscr{F}'$. Let $w = \vee \{j \;|\:  j \in C\}$. We claim that $w \in \mathscr{F}'$. Suppose on the contrary that $w \not \in \mathscr{F}'$. Hence $w$ contains a (finite) product of the  $i_{\alpha}$'s. say $\prod^{n}_{k=1} i_{k} \leq w $. This gives $\prod^{n}_{k=1} i_{k} \leq j$ for some $j \in C$, a contradiction to $j \in \mathscr{F}'$. Therefore  $w \in \mathscr{F}'$ and hence every chain $C$ in $\mathscr{F}'$  has an upper bound in $\mathscr{F}'$. Now, suppose that  every prime element $p$ of $L$ contains some  $i_{\alpha}$, i.e.  $p \in \mathscr{F}$. By Theorem \ref{2.5.}, a finite product of some of  $i_{\alpha}$ is zero element of $L$.
	\end{proof}

	\begin{cor}
		Let $L$ be a $c$-lattice and $S = \{ i_{\alpha}\} $, $\alpha \in \Lambda$ ($\Lambda$ is an indexed set) be the set of minimal prime elements of $L$. If each $ i_{\alpha}$ is compact element, then $|S| < \infty$. 
	\end{cor}
	
	\begin{proof}
		Since every prime element contains minimal prime element of $L$, by Theorem \ref{2.60.}, we have $\prod^{n}_{k=1} i_{k} =0 $. This gives $S \subseteq \{i_1, i_2, \cdots\cdots,i_n\}$ i.e.$|S| < \infty$.  
	\end{proof}
	
\end{document}